\documentclass[12pt, letterpaper]{article}
\usepackage[top=1.5cm, bottom=2cm, right=2cm, left=2cm]{geometry} 

\usepackage{color}
\usepackage{amsmath}
\usepackage{amsthm}
\usepackage{mathtools}
\usepackage{amssymb}
\usepackage{mathrsfs}
\usepackage{dsfont}
\usepackage{graphicx}
\usepackage{subcaption}
\usepackage[toc,page]{appendix}
\def \eps {\epsilon}
\usepackage{amsrefs}
 
\usepackage{xcolor}
\newtheorem{theorem}{Theorem}[section]
\newtheorem{corollary}[theorem]{Corollary}

\newtheorem{example}[theorem]{Example}
\newtheorem{lemma}[theorem]{Lemma}

\newtheorem{proposition}[theorem]{Proposition}

\newtheorem{question}{Question}[section]

\title{Chain Recurrence and Positive Shadowing in Linear Dynamics}
\author{Mayara Braz Antunes, Gabriel Elias Mantovani, Régis Varão   }

\newtheorem{maintheorem}{Theorem}

\begin{document}

\footnotetext{(M.B.A, G.E.M., R.V.) IMECC - UNICAMP, Departamento de Matematica, Rua Sérgio Buarque de
Holanda, 651, 13083-970 Campinas-SP, Brazil \\ \textit{E-mail addresses: ra163508@ime.unicamp.br, mantovani.gabriel@gmail.com, varao@unicamp.br }}

\maketitle

\begin{abstract}
    We study positive shadowing and chain recurrence in the context of linear operators acting on Banach spaces or even on normed vector spaces. We show that for linear operators there is only one chain recurrent set, and this set is a closed invariant subspace. We prove that every chain transitive linear dynamical system with positive shadowing property is frequently hypercyclic and, as a corollary, we obtain that every positive shadowing hypercyclic linear dynamical system is frequently hypercyclic. 
\end{abstract}



\section{Introduction}


Linear Dynamics is the study of linear operators on topological vector spaces. It is relatively simple to describe the dynamical behavior of any linear operator on a finite dimensional vector space. However, when the dimension of the vector space is infinite it turns out that the dynamical behavior of linear operators becomes very rich. To grasp how rich the dynamics can be we mention Feldman's result \cite{feldman} in which he proves that there is a linear operator $T:X \rightarrow X$ on a Banach space $X$ which ``contains'' all topological dynamical systems on compact spaces. More precisely, given a continuous map $f:M \rightarrow M$ on a compact metric space $M$, there is a compact $T$-invariant subset $Y$ in $X$ such that $f$ is conjugate to $T|_Y$.

In linear dynamics one is often interested in the study of the behavior  orbits of a system. For instance, one might ask whether a system has a dense orbit. In topological dynamics a point whose orbit is dense is called a \emph{transitive point}, while in linear dynamics we call such a point a \emph{hypercyclic vector}. A system which has a hypercyclic vector is called \textit{hypercyclic system}. Linear Dynamics is not solely influenced by Dynamical Systems, it is, of course, also influence by Functional Analysis. That is why definitions are not necessarily following those typically found in compact dynamics. A final remark of the importance of linear dynamics is the classical Functional Analysis open problem the ''Invariant Subspace Problem" (see \cite{dynamicsoflinearoperators}) which can be formulated in terms of closure of orbits of points, i.e. from a Dynamical System point of view.

There has been a recent effort to use tools that have helped dynamicists understand compact dynamical systems in the context of linear dynamical systems. As examples of such effort we may cite \cite{Messaoudi} and \cite{bernardes2020shadowing} that applied the concept of shadowing and hyperbolicity for linear dynamical systems, we may also cite \cite{brian} that used entropy in the study of translation operators. The main focus of this paper is to use the concepts of positive shadowing and chain recurrence in the study of linear dynamical operators.

We may say that both shadowing and chain recurrence study how the system responds to pseudo trajectories. In words a  pseudo trajectory (or chain) is almost a piece of orbit from the system. The difference from an actual orbit is that at each interaction of the dynamics there is a possible offset added to the result. This definition of chain occurs naturally in computational dynamics where given a dynamical system almost any orbit calculated in a computer will be a pseudo orbit, since few computer operations are error free. The formal definitions of chain recurrence and positive shadowing will be given in their respective sections. 



In this manuscript $X$ will always be a normed vector space (frequently a Banach space). We will denote by $\mathbb{K}$ the field over $X$, where $\mathbb{K} = \mathbb{R}$ or $\mathbb{K} = \mathbb{C}$. The symbol $\mathbb{N}$ denotes the set of natural numbers including $0$, that is, $\mathbb{N}=\{0,1,2,\ldots\}$.

We will now define some of the terms used in this article. Let $Y$ be a topological space and $f: Y \to Y$ a continuous function, $f$ is said to be \textbf{transitive} (or topologically transitive) if, for any pair of non-empty open sets $U, V \subset Y$, there is a natural number $N > 0$ such that $T^N U \cap V \neq \emptyset$. $f$ is said to be \textbf{topologically mixing} if, for any pair of non-empty open sets $U, V \subset Y$, there is $N \in \mathbb{N}$ such that, for every $n > N$, $f^{n}(U) \cap V \neq \emptyset$. It is a consequence of Birkhoff's Transitivity Theorem \cite{Birkhoff} that a linear dynamical system $(X,T)$, with $X$ a separable Banach space and $T : X \to X$ a linear operator, is transitive if, and only if, is hypercyclic.


An operator is frequently hypercyclic if it has a vector whose orbit visits each open set with a positive lower density. Formally, the \textbf{lower density} of a subset of natural numbers $A$ is defined by
\[\underline{dens}(A):=\liminf_{N\rightarrow\infty}\frac{\#(A\cap[1,N])}{N},\] where $\#(B)$ denotes the cardinality of the set $B$. A linear operator $T:X\rightarrow X$ on separable metric space $X$ is \textbf{frequently hypercyclic} if there is $x\in X$ such that
\[\underline{dens}(\{n\in\mathbb{N}\;:\;T^n(x)\in V\})>0\] for every non-empty open subset $V$ of $X$. 

Our main result is

\begin{maintheorem} \label{gymstheorem}
Let $X$ be a separable Banach space, and $T : X \to X$ an operator with both chain transitivity and positive shadowing property. Then $T$ is topologically mixing and frequently hypercyclic.
\end{maintheorem}

Since hypercyclic systems are chain transitive, then an immediate consequence of the previous theorem is the following corollary.

\begin{corollary} \label{hyperimpliesfreq.hyper}
Let $X$ be a separable Banach space. If $T : X \to X$ is hypercyclic and has the positive shadowing property then $T$ is frequently hypercyclic and topologically mixing.
\end{corollary} 





We also prove

\begin{maintheorem}\label{theo:nonwandering} \label{nonwanderingfreq}
Let $X$ be a normed vector space and $T$ a bounded linear operator on $X$ that has the positive shadowing property. Let $\Omega$ be the non wandering set of $T$. Then the following holds:

\begin{enumerate}
    \item $\Omega$ is the chain recurrent set of $T$;
    \item $\Omega$ is a closed and invariant subspace of $X$;
    \item Suppose further that $X$ is a separable Hilbert space and $T$ is self-adjoint. Then $T|_{\Omega} : \Omega \to \Omega$ is topologically mixing and frequently hypercyclic, in particular either $\Omega=\{0\}$ or $\Omega$ is infinite dimensional.
\end{enumerate}
\end{maintheorem}

In section 2 chain recurrent systems are defined and some elementary new results for the linear dynamical setting are obtained. In section 3 we define the concept of positive shadowing and use its synergy with chain recurrence to obtain the above theorems among other smaller results. The last section brings some open questions that emerged during the elaboration of this text to motivate future research.


\section{The Chain Recurrent Subspace}

Let $(Y,d)$ be a metric space and $f: Y \rightarrow Y$ a continuous function. We say that a finite sequence $\{x_0, x_1,\ldots, x_n\}$ is an \textbf{$\epsilon$-chain} with $\epsilon>0$ if $n\in\mathbb{N}\setminus\{0\}$ and $d(f(x_i),x_{i+1})< \epsilon$ for every $ 0 \leq i < n$. Given two points $x,y\in Y$, we write $x\mathcal{R}y$ if the following holds:
\begin{center}
     given $\epsilon>0$ there is an $\epsilon$-chain beginning in $x$ and ending in $y$, $\{x_0=x,x_1,\ldots, x_{n-1},x_n=y\}$, and another beginning in $y$ and ending in $x$, $\{y_0=y, y_1,\ldots, y_{m-1},y_m=x\}$.
\end{center} The set $CR(f)=\{x\in Y\;:\;x\mathcal{R}x\}$ is called \textbf{chain recurrent set}. A point $x\in CR(f)$ is called a \textbf{chain recurrent point}. Restricted to $CR(f)$, the relation $\mathcal{R}$ is an equivalence relation. We will say that a dynamical system $(X,f)$ is \textbf{chain transitive} if $x\mathcal{R} y$ for every $x,y\in Y$. 


As an example of chain transitive system one may easily see that the identity operator on any normed vector space is chain transitive. Also any hypercyclic operator on any normed vector space is chain transitive as well. In compact dynamics it is a relevant problem to find how many different chain recurrent classes there are in the space. The next two results state that a linear dynamical system has only one chain recurrent class, which we will simply call the chain recurrent set. 

\begin{theorem} \label{spanofchainrecurrent}
 Let $X$ be a normed vector space and $T$ be a bounded linear operator acting on $X$. If $x \in X$ is chain recurrent, then every point of span$[x]$ is chain recurrent and span$[x]$ is contained in only one recurrent class.
\end{theorem}
\begin{proof}
Since $x$ is chain recurrent, given $\eps> 0$ there is an $\eps$-chain \[\{x_0=x, x_1, \ldots, x_n=x\},\] that is, $\|Tx_0 - x_1 \| < \eps,\ldots,\| T x_{n-1} - x \| < \eps$. One can readily verify that if $|\lambda| \in  (0,1]$ 
then the finite sequence $\{\lambda x_0, \lambda x_1, \ldots, \lambda x_n\}$ is a finite $\eps$-chain that begins and ends in $\lambda x$. Therefore $\lambda x$ is chain recurrent for any $|\lambda| \in (0,1]$. We know that $0$ is always chain recurrent, since $T(0)=0$, therefore $\lambda x$ is chain recurrent for $|\lambda| \in [0,1]$.

The case $|\lambda| > 1$ is analogous: given $\eps > 0 $ we know that there is an $\eps/ |\lambda|$-chain starting and ending in $x$, given by $\{x_0, x_1, \ldots ,x_n\}$, it is not hard to see that the chain given by $\{\lambda x_0, \lambda x_1, \ldots , \lambda x_n\}$ is an $\eps$-chain starting and ending in $\lambda x$.



We have proved above that every point of the span$[x]$ is chain recurrent. It remains to prove that span$[x]$ belongs to just one recurrent class. Given $\lambda \in \mathbb{K}$ and $\eps>0$, we want to find an $\eps$-chain that begins in $x$ and ends in $\lambda x$. Let $k \in \mathbb{N}$ be such that $\displaystyle\dfrac{\|x-\lambda x\|}{k}<\frac{\eps}{2}$. For each $j\in\{0,\ldots, k\}$, consider 
$$
\displaystyle x^j=\left(1-\frac{j}{k}\right)x+\frac{j}{k}\lambda x \in \mbox{span}[x].
$$ 

We proved above that $x^j$ is chain recurrent. Now we can choose for each $j\in\{0,1,\ldots, k-1\}$ an $\eps/2$-chain $\{x_0^j=x^j, x_1^j,\ldots, x_{{n_j}-1}^j, x_{n_j}^j=x^j\}$. The sequence $$\{x_0^0=x, x_1^0,\ldots, x_{n_0-1}^0, x_0^1=x^1, x_1^1,\ldots, x_{n_1-1}^1,\ldots, x_0^{k-1}=x^{k-1},x^{k-1}_1,\ldots, x_{n_{k-1}-1}^{k-1}, x^k=\lambda x\}$$ is an $\eps$-chain from $x$ to $\lambda x$. Indeed, it is enough to show that $$\|Tx_{n_j-1}^j-x^{j+1}\|<\eps,\;\forall\;j\in\{0,1,\ldots, k-1\}.$$ This follows from the fact that
$$\|x^j-x^{j+1}\| = \left\|\left(1-\frac{j}{k}\right)x+\frac{j}{k}\lambda x -  \left(1-\frac{j+1}{k}\right)x-\frac{j+1}{k}\lambda x\right\| = \frac{1}{k}\|x-\lambda x\|<\frac{\eps}{2}$$ and
$$\|Tx_{n_j-1}^j-x^{j+1}\|\leq \|Tx_{n_j-1}^j - x^j\|+\|x^j-x^{j+1}\|<\dfrac{\eps}{2}+\dfrac{\eps}{2}= \eps,$$ for all $j\in\{0,\ldots k-1\}$.
We now need to do the inverse path, go from $\lambda x$ to $x$. We saw that $\lambda x$ is chain recurrent. Then, if $\lambda\neq 0$ we can apply the previous argument to create an $\eps$-chain from $y=\lambda x$ to $\lambda'y=x$ with $\lambda'=\frac{1}{\lambda}$. If $\lambda=0$, that is, $\lambda x=0$, consider ${\lambda'}\neq 0$ such that $\|{\lambda'}x\|\leq\epsilon/2$ and use the previous argument to create an $\epsilon/2$-chain, $\{x_0=\lambda'x,x_1,\ldots, x_{n-1}, x_n= x\}$ which begins in $\lambda'x$ and ends in $x$. This yields an $\eps$-chain, $\{0,x_0=\lambda'x,x_1,\ldots, x_{n}=x\}$, starting in the origin and ending in $x$. 
\end{proof}

\begin{corollary} \label{onerecurrenceclass}
If $X$ is a normed vector space and $T$ is a bounded linear operator acting on $X$, then $T$ has only one chain recurrent class.
\end{corollary}

\noindent \textbf{Proof:} It is clear that $0$ is chain recurrent for any linear operator $T$, therefore the chain recurrent set of $T$ is non empty. Now, due to the previous result, every chain recurrent class has the origin of $X$ in common. Since these chain recurrent classes have a common point they are the same.

\qed

In view of Corollary \ref{onerecurrenceclass}, we will refer to the chain recurrent class of an operator as the chain recurrent set. This corollary says that a bounded linear operator is chain transitive if every point of the space is chain recurrent. Notice that the chain recurrent set of any operator is non-empty since the origin is always contained in this set.

Let $Y$ be a set and $f : Y \to Y$ a function. A subset $A$ of $Y$ is \textbf{invariant} for $f$ if $f(A) \subset A$. The next corollary tell us that the chain recurrent set is a closed invariant subspace.

\begin{corollary} \label{chainspace}
The chain recurrent set is a closed and invariant subspace. If $T$ is invertible, $T(CR(T))= CR(T)$.
\end{corollary}

\noindent \textbf{Proof: }  Theorem \ref{spanofchainrecurrent} tell us that $CR(T)$ is closed under scalar multiplication. Let $x,y \in CR(T)$, $\epsilon > 0$ and $\{x_0=x,x_1,\ldots,x_n=0\}$, $\{y_0=y,y_1,\ldots,y_m=0\}$ be two $\epsilon/2$-chains that go from $x$ to $0$ and from $y$ to $0$ respectively. We may suppose $m>n$ so $$\{x+y=x_0+y_0,x_1+y_1, \ldots , x_n + y_n, 0 + y_{n+1} , \ldots ,0 + y_{m} = 0\}$$ is an $\epsilon$-chain that connects $x+y$ to $0$. A similar idea may be used to show that one can go from zero to $x+y$ with an $\epsilon$-chain. Therefore $x+y$ is chain recurrent.

The fact that the chain recurrent set is closed and invariant is true for topological dynamical systems in general. But we include the proof of these facts for the reader's convenience.

To show that $CR(T)$ is closed consider $\{x_n\}$ a sequence in $CR(T)$ converging to a point $x$ in $X$. Given $\epsilon > 0$ choose $x_n$ such that $\|x_n - x\| < \epsilon/2$ and $\|Tx_n-Tx\|<\eps/2$ and let $\{y_0=x_n,y_1,\ldots,y_N=x_n\}$ be an $\epsilon/2$-chain that goes from $x_n$ to $x_n$. It is immediate to see that $\{x,y_1,\ldots, y_{N-1},x\}$ is an $\epsilon$-chain that goes from $x$ to $x$. Therefore, $x\in CR(T)$.

We now prove the invariance. Given $x \in CR(T)$, there is an $\epsilon/(2 \max \{\|T\|,1\})$-chain, $\{x_0=x,x_1,\ldots,x_N=x\}$, that goes from $x$ to $x$. Define the chain $\{y_0=Tx, y_1=x_2 , y_2=x_3 , \ldots, y_{N-1} = x_N =x,y_N=Tx\}$. The only difficult step is to prove that $\|Ty_0 - y_1\|$ is smaller than $\eps$. But we have that
$$\begin{array}{rcl}
\|Ty_0 - y_1\|&=&\|Ty_0 - x_2\|\\
\\
&=&\|T^2 x -Tx_1 + Tx_1 - x_2\|\\
\\
&\leq& \|T^2 x -Tx_1\| + \|Tx_1 - x_2\|\\
\\
&\leq&\|T\| \, \|T x -x_1 \| + \|Tx_1 - x_2\|\\
\\
&<&\epsilon.
\end{array}
$$ 

Suppose now that $T$ is invertible. To show that $T(CR(T))=CR(T)$ it is enough to prove that $T^{-1}(CR(T))\subset CR(T)$ since it is already proven that $T(CR(T))\subset CR(T)$. Let $x\in CR(T)$ and $\eps>0$. Then, there is an $\epsilon/(2 \max \{\|T^{-1}\|,1\})$-chain, $\{x_0=x,x_1,\ldots,x_N=x\}$, that goes from $x$ to $x$. The finite sequence $\{y_0=T^{-1}x, y_1=x_0=x , y_2=x_1 , \ldots, y_{N-1} = x_{N-2}, y_{N}=T^{-1}x\}$ is an $\eps$-chain from $T^{-1}x$ to $T^{-1}x$, since 
\[\begin{array}{rcl}
    \|T(y_{N-1})-y_N\| & = &\|T(x_{N-2})-T^{-1}(x)\|\\
    \\
     & \leq &\|T(x_{N-2})-x_{N-1}\|+\|x_{N-1}-T^{-1}(x)\|\\
     \\
     &\leq&\|T(x_{N-2})-x_{N-1}\|+\|T^{-1}\|\|T(x_{N-1})-x\|\\
     \\
     &<&\eps.
\end{array}\]

\qed

The next propositions give examples of operators that are chain transitive and operators that are not.


\begin{proposition} \label{unitaryimplieschaintransitivity}
Let $X$ be a normed vector space and $T:X \to X$ a bounded linear operator which is a surjective isometry (or equivalently, T is invertible and $\|T^{-1}\| = \|T\| = 1$), then $T$ is chain transitive. In particular, if $X$ is an inner product space, every unitary operator $T:X \to X$ is chain transitive.
\end{proposition}

\noindent \textbf{Proof:} Let $x\in X$ and $\epsilon > 0$. We will show that $x$ is a chain recurrent point. Choose $n \in \mathbb{N}$ such that $\dfrac{\|x\|}{n} < \dfrac{\epsilon}{2}$. Define the sequence
$$\begin{array}{rcl}
x_0 & = & x,\\
\\
x_1& = &T(x) + \dfrac{T^{-n+1}(x)}{n} - \dfrac{T(x)}{n},\\
\\
x_2 & = &T(x_1) + \dfrac{T^{-n+2}(x)}{n} - \dfrac{T^2(x)}{n} \;\;=\;\; T^2(x) + \dfrac{2T^{-n+2}(x)}{n} - \dfrac{2T^2(x)}{n},\\
&\vdots&\\
x_k &= &T(x_{k-1}) + \dfrac{T^{-n+k}(x)}{n} - \dfrac{T^k(x)}{n} \;\;=\;\; T^k(x) + \dfrac{kT^{-n+k}(x)}{n}- \dfrac{kT^k(x)}{n},\\
\end{array}$$

for every $1 \leq k \leq n$. Notice that $x_n = x$, and that 
$$
\|x_k - T(x_{k-1})\| = \left|\left| \dfrac{T^{-n+k}(x)}{n} - \dfrac{T^k(x)}{n}\right|\right| \leq \left| \left| \dfrac{T^{-n+k}(x)}{n} \right| \right| +\left| \left| \dfrac{T^{k}(x)}{n} \right| \right| < \epsilon
$$
for $1 \leq k \leq n$. Therefore, $\{x_0,x_1,\ldots,x_n\}$ is an $\eps$-chain which begins and ends in $x$. This means that $x$ is a chain recurrent point.

\qed

If $X$ is a normed vector space, a map $T: X \to X$ is \textbf{recurrent} if, for every $x \in X$, and for every open set $U \subset X$, with $x \in U$, there is some $n \in \mathbb{N} \setminus \{0\}$, such that $T^n(U) \cap U \neq \emptyset$. Clearly any recurrent operator is also chain transitive, but there are operators that are chain transitive and not recurrent. Indeed let $X = \ell_p(\mathbb{Z})$ for $1 \leq p \leq \infty$ and $T: X \to X$ be the shift $T(e_i) = e_{i-1}$. Then by the above proposition $T$ is chain transitive. For $x = e_0$ there is no $n>0$ such that $T^n(B(x,1/2)) \cap B(x,1/2) \neq \emptyset$, therefore $T$ is not recurrent.






Let $X$ be a normed vector space. We say that a bounded linear operator $T: X \to X$ is a \textbf{proper contraction} if $\|T\| < 1$ and a \textbf{contraction} if $\|T\| \leq 1$. We say that $T$ is a \textbf{proper dilation} if $T$ is invertible and $\|T^{-1}\|<1$ and a \textbf{dilatation} if $\|T^{-1}\| \leq 1$. An operator $T$ on a Banach space $B$ is said to be \textbf{hyperbolic} \cite{hyperbolic} if there is a splitting
$$
B = B_s \oplus B_u, \hspace{1cm} T = T_s \oplus T_u,
$$
where $B_s$ and $B_u$ are closed $T$-invariant linear subspaces of $B$, $T_s = T |_{B_{s}}$ is a proper contraction, and $T_u = T |_{B_u}$ is a proper dilation. It is common in the literature to assume that hyperbolic operators are invertible, but in this manuscript such assumption is not needed and therefore we do not assume it.

\begin{proposition}\label{CRtrivial}
Let $X$ be a normed vector space and $T : X \to X$ a linear operator. If $T$ is a proper contraction, then the chain recurrent set of $T$ is the origin.
\end{proposition}

\noindent \textbf{Proof: } Recall that $0$ is always in the chain recurrent set. Let $x\neq 0$, define $\epsilon = (\|x\| - \|Tx\|-\delta)(1-\|T\|)$ and choose a small enough $\delta >0$ such that $\eps > 0$. Consider $\{x_0=x,x_1,\ldots,x_N\}$ an $\epsilon$-chain that starts in $x$. Thus, we have that 
$$\begin{array}{rcl}
\|x_{N}\| &=&
\|x_{N} +  (Tx_{N-1} - Tx_{N-1}) +  \cdots + (T^{N-1}x_1 - T^{N-1}x_1) + (T^{N}x_0 - T^{N}x_0) \|\\ 
\\
&=&\|(x_{N} - Tx_{N-1}) + T(x_{N-1} - Tx_{N-2}) + \cdots + T^{N-1}(x_1 - Tx) + T^{N}x\|\\ 
\\
&\leq & \|(x_{N} - Tx_{N-1})\| + \|T\| \, \|x_{N-1} - Tx_{N-2}\| + \cdots + \|T^{N-1}\| \, \|x_1 - Tx\| + \|T^{N}x\|\\
\\
&\leq&\|T^N x\| + \dfrac{\epsilon}{1-\|T\|}\\
\\
&=&\|x\|-(\|Tx\|-\|T^Nx\|+\delta)\\
\\
&<& \|x\|,
\end{array}
$$
since $\|T^Nx\|\leq\|Tx\|$. Therefore there is no $\epsilon$-chain that starts in $x$ and finishes in $x$. 

\qed

\begin{lemma} \label{refereelemma}
If $T$ is an invertible operator on the normed space $X$, then $CR(T)=CR(T^{-1})$.
\end{lemma}

\noindent \textbf{Proof:}
Let $x\in CR(T)$, $\epsilon>0$ and $\{x_0=x,\ldots, x_n=x\}$ an $\epsilon/\|T^{-1}\|$-chain for $T$. It is then straightforward that $\{y_0=x, y_1=x_{n-1},\ldots, y_{n-1}=x_1,y_n=x\}$ is an $\epsilon$-chain for $T^{-1}$ from $x$ to $x$. This means that $CR(T)\subset CR(T^{-1})$ and interchanging the roles of $T$ and $T^{-1}$ we get the conclusion.
\qed

\begin{corollary} \label{properdilation}
If $T$ is a proper dilation, then $CR(T)=\{0\}$.
\end{corollary}

\noindent \textbf{Proof}
It follows immediately from the above lemma and Proposition \ref{CRtrivial}.

\qed

The next theorem is interesting in itself and will provide two corollaries. Corollary \ref{decompositionchainrecurrentset} will aid us to prove that $CR(T)=\{0\}$ when $T$ is hyperbolic. In \cite{fabricio} it is proved something weaker, that $T$ is not chain recurrent when $T$ is hyperbolic. Corollary \ref{chainrecurrentselfadjoint} will aid in the proof of Theorem B.




\begin{theorem}\label{CRprojection}
Let $T$ be a bounded operator on a Banach space $X$ such that $X=M\oplus N$ where $M,N$ are $T$-invariant closed subspaces of $X$. Then $CR(T)\cap M=CR(T|_M)$. 
\end{theorem}


\noindent \textbf{Proof:} It is clear that $CR(T|_{M})\subset CR(T) \cap M$. We will show that $CR(T)\cap M\subset CR(T|_{M})$. By Theorem 2.10 of \cite{Brezis}, there is $\alpha>0$ such that \begin{equation}\label{brezisequation}\|m\|\leq\alpha\|m+n\| \mbox{ and }\|n\|\leq\alpha\|m+n\|\end{equation} whenever $m\in M$ and $n\in N$. Given $x\in CR(T)\cap M$ and $\eps>0$ there is an $\eps/\alpha$-chain, $\{x,x_1,x_2,\ldots,x_{n-1},x\}$, from $x$ to $x$. For each $i\in\{1,2,\ldots,n-1\}$ there are $z_i\in M$ and $w_i\in N$ such that $x_i=z_i+w_i$. We shall show that $\{x,z_1,z_2,\ldots,z_{n-1},x\}$ is an $\eps$-chain in $M$ from $x$ to $x$, which will complete the proof. Since $M$ and $N$ are $T$-invariant and $X=M\oplus N$, by (\ref{brezisequation}) we have that
\[\begin{array}{l}
\|Tx-z_1\|\leq \alpha\|Tx-(z_1+w_1)\|<\alpha\dfrac{\eps}{\alpha}=\eps,\\
\\
\|Tz_i-z_{i+1}\|\leq \alpha\|T(z_i+w_i)-(z_{i+1}+w_{i+1})\|<\alpha\dfrac{\eps}{\alpha}=\eps,\;\;\forall\; i\in\{1,2,\ldots,n-2\}, and\\
\\
\|Tz_{n-1}-x\|\leq \alpha\|T(z_{n-1}+w_{n-1})-x\|<\alpha\dfrac{\eps}{\alpha}=\eps.
\end{array}\] This guarantees that $\{x,z_1,z_2,\ldots,z_{n-1},x\}$ is an $\eps$-chain in $M$ beginning and ending in $x$, that is, $x\in CR(T|_{M})$.

\qed

\begin{corollary}\label{decompositionchainrecurrentset}
Let $T$ be a bounded operator on a Banach space $X$. Suppose that $X = M \oplus N$, where $M$ and $N$ are closed $T$-invariant subspaces of $X$. Then $CR(T)=CR(T|_M)\oplus CR(T|_N)$.
\end{corollary}

\noindent \textbf{Proof:} Proposition \ref{chainspace} and the fact that $M$ and $N$ are closed and invariant subspaces of $X$ give us that $CR(T) \cap M$ and $CR(T) \cap N$ are closed invariant subspaces of $X$ and therefore of $CR(T)$. We know that every element $x \in CR(T)$ can be written in a unique manner as $x=m+n$, with $m \in M$ and $n \in N$. Following a similar reasoning as in the proof of Theorem \ref{CRprojection}, we have that both $m$ and $n$ are chain recurrent. Therefore $m \in CR(T) \cap M$ and $n \in CR(T) \cap N$, since $x$ is arbitrary, this implies that
\begin{equation} \label{projecaodoCR}
CR(T) = (CR(T) \cap M) \oplus (CR(T) \cap N)   . 
\end{equation}
By Theorem \ref{CRprojection} the above expression may be written as
$$
CR(T) = CR(T|_M) \oplus CR(T|_N).
$$

\qed





\begin{corollary}\label{chainrecurrentselfadjoint}
Let $T$ be a self-adjoint bounded operator on a Hilbert space $X$. Then the chain  recurrent set of $T|_{CR(T)}$ coincides with $CR(T)$. In particular, $T|_{CR(T)}$ is chain transitive.
\end{corollary}

\noindent \textbf{Proof:} Since $X$ is a Hilbert space and $CR(T)$ is a closed space, then $X=CR(T)\oplus (CR(T))^\perp$. We have already seen that $CR(T)$ is invariant for $T$ which implies that $(CR(T))^\perp$ is invariant for $T^*=T$. By Theorem \ref{CRprojection} $CR(T|_{CR(T)}) = CR(T)$.

\qed

The following results are immediate consequences of Corollary \ref{decompositionchainrecurrentset}.


\begin{corollary}
Let $T$ be a hyperbolic operator on a Banach space $X$ then $CR(T) = \{0\}$.
\end{corollary}

\noindent \textbf{Proof:} It follows immediately from the corollaries  \ref{decompositionchainrecurrentset}, \ref{properdilation} and proposition \ref{CRtrivial}.

\qed

\begin{corollary}\label{CRforsubspaces}
Let $T$ be an operator on a Banach space $X$. Suppose that $X = M \oplus N$, where $M$ and $N$ are closed $T$-invariant subspaces of $X$. Then $T$ is chain transitive if, and only if, $T|_M$ and $T|_N$ are both chain transitive.
\end{corollary}

The following example shows that $CR(T)$ can be non-trivial.

\begin{example} \label{refereeexample}
Let $X$ be a non-trivial normed space, $T:X\rightarrow X$ a proper contraction and $I:X\rightarrow X$ the identity operator. The operator
\[T\times I:X\times X\rightarrow X\times X,\] where $X\times X$ is endowed with any of the typical product norms, satisfies that $CR(T\times I)=\{0\}\times X$.
\end{example}

The next result, which is used in the last section of this paper, asserts that chain transitivity is preserved under Cartesian product. This is an obvious consequence of Corollary \ref{decompositionchainrecurrentset} when one assumes that the spaces are Banach.

\begin{proposition} \label{multiplechainrecurrence}
Let $T_1, T_2, ..., T_k$ be bounded chain transitive operators on normed vector spaces $X_1, X_2,..., X_k$ respectively. Then the product $T_1 \times ... \times T_k :X_1\times  ... \times X_k \rightarrow X_1 \times ... \times X_k$ is chain transitive.
\end{proposition}

\noindent \textbf{Proof:} Consider any of the typical product norms in the product space. We will prove the case when $k=2$, the proof of the general case follows by induction. Let $(x,y) \in X_1 \times X_2$. Given $\eps>0$, consider an $\eps$-chain $\{x_0=x,x_1,\ldots,x_n=0\}$ that connects $x$ with $0$. We can see that \[\{(x_0,y),(x_1,T_2(y)),(x_2,{T_2}^2(y)),\ldots,(x_n, {T_2}^n(y))\}\] is an $\eps$-chain that connects $(x,y)$ with $(0,{T_2}^n(y))$. Following the same reasoning we are able to create an $\eps$-chain that connects $(0,{T_2}^n(y))$ with $(0,0)$. 
To go from $(0,0)$ to $(x,y)$ we consider $\eps/2$-chains $\{x_0=0,x_1,\ldots,x_n=x\}$ in $X_1$ and $\{y_{0}=0,y_1,\ldots,y_m=y\}$ in $X_2$ that go from $0$ to $x$ and from $0$ to $y$ respectively. If $n=m$, the sequence $\{(x_0,y_0), (x_1,y_1),\ldots, (x_n,y_n)\}$ is an $\eps$-chain beginning in $(0,0)$ and ending in $(x,y)$. Assuming $n>m$, then the finite sequence $$\{(0,0), (x_1,0), (x_2,0), \ldots,(x_{n-m},0),(x_{n-m+1},y_{1}), (x_{n-m+2},y_{2}), \ldots, (x_n,y_m) = (x,y)\}$$ is an $\eps$-chain connecting $(0,0)$ to $(x,y)$. 
One can see that this finite sequence is an $\eps$-chain that connects $(0,0)$ with $(x,y)$.


\qed

\section{Chain Recurrence and Positive Shadowing}

Let $(Y,d)$ be a metric space, $f : Y \to Y$ a continuous function, $\delta>0$ and $\{x_n\}_{n \in \mathbb{N}}$ a sequence in $Y$. We say that $\{x_n\}_{n \in \mathbb{N}}$ is a \textbf{positive $\delta$-pseudo orbit} of $f$ if $d(f(x_n),x_{n+1}) \leq \delta$ for all $n \in \mathbb{N}$. The function $f$ have the \textbf{positive shadowing property} if for each $\epsilon > 0$ there is $\delta > 0$ such that every $\delta$-pseudo orbit is $\epsilon$-shadowed by an $x \in Y$, i.e., there is $x \in Y$ such that 
$$
d(x_n,f^{n}(x)) < \epsilon \, \, \text{ for all }n  \in \mathbb{N}.
$$

It is easy to see that the identity map does not have the shadowing property, on the other hand it is also easy to see that proper contractions and proper dilations have the shadowing property. More generally, any hyperbolic operator on a Banach space has the shadowing property \cite{Messaoudi}. In \cite{bernardes2020shadowing} one necessary and sufficient condition for the weighted shift to have the shadowing property is obtained. To illustrate the positive shadowing property we present two examples bellow. The first one is an operator that has an eigenvalue equal to $1$ (and therefore it is not hyperbolic) and has positive shadowing. The second example is an invertible contraction that does not have the positive shadowing property.

\begin{example} \label{eigen1andshadowing}
In the space $\ell_p(\mathbb{N})$, for $1 \leq p \leq \infty$, consider the operator $T: \ell_p(\mathbb{N}) \to \ell_p(\mathbb{N})$ given by $T(e_0) = e_0$, and $T(e_i) = 2 e_{i - 1}$ if $i \geq 1$, where $\{e_i\}_{i\in\mathbb{N}}$ is the canonical basis. Note that there is an eigenspace associated with the eigenvalue $1$. This operator has positive shadowing. To see this consider the operator $Se_{i} = e_{i+1}/2$ for every $i \in \mathbb{N}$. Note that $\|S\|=1/2$ and that $S$ is a right inverse for $T$. Given any $\delta$-pseudo orbit $\{x_n\}_{n \in \mathbb{N}}$ note that the point 
$$
x = x_0 + S(x_1 - Tx_0) + S^{2}(x_2 - Tx_1) + \cdots
$$
is in $\ell_p(\mathbb{N})$ and $2 \delta$-shadows the pseudo-orbit. Therefore $T$ has positive shadowing. 

It is not hard to see that $T$ is transitive and therefore is chain transitive when $1 \leq p < \infty$. Indeed let $U$ and $V$ be two non-empty open subsets of $\ell_{p}(\mathbb{N})$. Since sequences with a finite number of non null elements are dense in $\ell_{p}(\mathbb{N})$ let $x=(x_0,x_1,\ldots,x_n,0,0,\ldots) \in U$ and $y=(y_0,y_1,\ldots,y_m,0,0,\ldots) \in V$. Consider 
$$
r = \sum_{i=0}^{n} x_i 2^{i}
$$
and choose $k,l \in \mathbb{N}$ such that 
$$
z = ( 
\underbrace{
\underbrace{x_0,x_1,\ldots,x_n,0,0,\ldots, 0 , -\dfrac{r}{2^k}}_{k \text{  positions}}, 0, \ldots, 0, \dfrac{y_{0}}{2^{l}}}_{l \text{ positions}}, \dfrac{y_{1}}{2^{l+1}}, \ldots, \dfrac{y_{m}}{2^{l+m}},0, 0, \ldots ) \in U.
$$
Hence $T^lz = y$ and so $T^{l} U \cap V \neq \emptyset$, therefore $T$ is transitive. Since $T$ has positive shadowing and is transitive, then Theorem A gives us that $T$ is topologically mixing and frequently hypercyclic.
\end{example}


\begin{example}
Consider the space of real Lebesgue integrable functions in $[1/2,1]$, $L_1([1/2,1])$, and the operator $T : L_1([1/2,1]) \to L_1([1/2,1])$, given by $T(f)(x)=f(x)x$ for every $x \in [1/2,1]$. Notice that $T$ is invertible. We will prove that this operator does not have the shadowing property, that is, for each $\delta>0$ there is a $\delta$-pseudo orbit which cannot be $\eps$-shadowed for any $\eps>0$. Given $\delta>0$ consider the sequence $\{f_n^\delta\}_{n\in\mathbb{N}}$ where
\[f_n^\delta(x)=\delta(x^{n-1}+x^{n-2}+\cdots+x+1).\] This sequence is a $\delta$-pseudo orbit, since
\[\begin{array}{rcl}\|T(f_n^\delta)-f_{n+1}^\delta\|_1 & = & \|\delta(x^{n-1}+x^{n-2}+\cdots+x+1)x - \delta(x^n+x^{n-1}+\cdots+x+1)\|_1\\
\\
&=&\|\delta(x^{n}+x^{n-1}+\cdots+x^2+x) - \delta(x^n+x^{n-1}+\cdots+x+1)\|_1\\
\\
&=& \|\delta\|_1 \;=\;\displaystyle\int_{1/2}^1\delta dx\;=\;\dfrac{\delta}{2}\;<\;\delta.
\end{array}\]
Let us show that $\{f_n^\delta\}$ previously defined cannot be $\eps$-shadowed for any $\eps>0$. Indeed, for any $g\in L^1([1/2,1])$ we have that
$$\begin{array}{rcl}
\|T^{n}g - f_n^{\delta}\|_1& = &
\|x^n g(x)  - \delta (x^{n-1} + \cdots + 1)\|_1\\
\\
&\geq&| \, \|x^ng(x) \|_1 - \|\delta (x^{n-1} + \cdots + 1)\|_1 \, |.
\end{array}
$$
Notice that since $x<1$ almost surely, then $x^n g(x) \to 0$ almost surely as $n \to \infty$. Therefore, by the Dominated Convergence Theorem, we have that $\|x^ng(x) \|_1 \to 0$ as $n \to \infty$ and so this term is bounded. The second term in the above sum is not bounded. Indeed,
$$\begin{array}{rcl}
\displaystyle\|\delta (x^{n-1} + \cdots + 1)\|_1 &=& \displaystyle\delta \int_{1/2}^{1} x^{n-1} + \cdots + 1dx\\
\\
&= &\displaystyle
\delta \left[ \dfrac{x^n}{n} + \dfrac{x^{n-1}}{n-1} + \cdots + x \right]_{1/2}^1\\
\\
&= &\displaystyle\delta \left( \dfrac{1}{n} + \dfrac{1}{n-1} + \cdots + 1 - \dfrac{1}{2^n} \dfrac{1}{n} - \dfrac{1}{2^{n-1}} \dfrac{1}{n-1} - \cdots - \dfrac{1}{2} \right)\\
\\
&=& \displaystyle \delta \sum_{k=1}^{n} \dfrac{1}{k}\left( 1 - \dfrac{1}{2^k} \right),
\end{array}
$$
using a comparison test with the harmonic series, for instance the limit comparison test \cite{limit}, one can readily see that the right side goes to $+\infty$ as $n \to + \infty$. This ensures that $\|T^ng-f_n\|_1\rightarrow\infty$, therefore the $\delta$-pseudo orbit $\{f_n^\delta\}_{n\in\mathbb{N}}$ cannot be $\eps$-shadowed for any $\eps>0$.

\end{example}

It turns out that chain transitivity and positive shadowing have an interesting synergy that allows us to obtain the most important results of this paper. Chain transitivity allows us to connect different points of the space with $\epsilon$-chains and the shadowing property tell us that there will be a point that shadows such chain. 


\begin{theorem} \label{positiveshadowingmixing}
Let $X$ be a normed vector space, and $T : X \to X$ a bounded linear operator with both chain transitive and positive shadowing property. Then $T$ is topologically mixing.
\end{theorem}

\noindent \textbf{Proof:} Let $U$, $V$ be non-empty open subsets of $X$. Let $x$ be a vector in $U$, and $y$ a vector in $V$ and let $\lambda > 0$ be such that $B(x,\lambda) \subset U $ and $B(y,\lambda) \subset V$. Let $\epsilon = \lambda/2$ and let $\delta>0$ be associated with this $\epsilon$ from the positive shadowing property. Since $T$ is chain transitive, there is a $\delta$-chain that goes from $x$ to the origin, $\{x_0,x_1,\ldots,x_n\}$, and a $\delta$-chain that goes from the origin to $y$, $\{y_0,y_1,\ldots,y_m\}$. Then we have that 
$$
\{x, x_1, x_2, \ldots, x_n, \underbrace{0, 0, \ldots ,0}_{\text{total of } k \text{ zeroes}} ,y_1, \ldots, y_{m-1}, y, Ty, T^2y, \ldots\}
$$
is a positive $\delta$-pseudo orbit for every $k \in \mathbb{N}$, and therefore there is a $z_k \in X$ that $\epsilon$-shadows such pseudo orbit. Notice that $z_k \in U$ and $T^{n+m+k}z_k \in V$ for every $k \in \mathbb{N}$, hence $T^{n+m+k}U \cap V \neq \emptyset$ for every $k \in \mathbb{N}$.

\qed

The above Theorem together with Proposition \ref{multiplechainrecurrence} provide the following corollary.

\begin{corollary}
If $T_i : X_i \to X_i$ for $1\leq i \leq n$ is a finite family of linear dynamical systems such that, for each $i$, $(T_i,X_i)$ satisfies the hypothesis of the Theorem \ref{positiveshadowingmixing}. Then $T_1 \times \ldots \times T_n : X_1 \times \ldots X_n \to X_1 \times \ldots X_n$ is topologically mixing.
\end{corollary}

The next lemma allows us to obtain subsets of $\mathbb{N}$ that are big enough to have positive lower density and at the same time are spread enough apart. This lemma is crucial in the proof of our main result, Theorem \ref{gymstheorem}. 

\begin{lemma}\cite[Lemma 6.19]{dynamicsoflinearoperators}\label{freqhypercycliclemma}
Let $\{N_p\}_{p \geq 1}$ be any sequence of positive real numbers. Then one can find a sequence $\{\Delta_p\}$ of pairwise disjoint subsets of $\mathbb{N}$ such that 
\begin{enumerate}
    \item Each set $\Delta_p$ has positive lower density;
    \item $\min(\Delta_p) \geq N_p$ and $|n-m| \geq N_p + N_q$; whenever $n \neq m$ and $(n,m) \in \Delta_p \times \Delta_q$.
\end{enumerate}
\end{lemma}


\noindent \textbf{Proof of Theorem \ref{gymstheorem}:} 

The conclusion of being topologically mixing is a direct consequence of Theorem \ref{positiveshadowingmixing}, we only need to prove that $T$ is frequently hypercyclic. Let $\{x_p\}_{p \in\mathbb{N}}$ be a countable dense set of vectors in $X$. For every $p \in \mathbb N$, let $\epsilon_p = 1/2^p$. By the positive shadowing property for each $p\in\mathbb{N}$  there is a $\delta_p>0$ such that every $\delta_p$-pseudo orbit is $\epsilon_p$ shadowed by some point of $X$. 

For each $x_p$ in the dense countable set let $N_p$ be the size of an $\delta_p/2$-chain that connects $0$ to $x_p$ and then to $0$ again (by the size of a chain we mean its cardinality). By completing with zeroes, we may suppose $\{N_p\}_{p \in \mathbb{N}}$ is a strictly increasing sequence of even numbers of the form $N_p=2R_p$, with $R_p\in\mathbb{N}\setminus\{0\}$ and ${1}/{R_p}<{\delta_p}/{4}$ for every $p\in\mathbb{N}$. We may also suppose that it takes half of the chain to reach $x_p$ from $0$ and other half to go from $x_p$ to $0$. Formally, our chain has the general form 
\begin{equation} \label{z0pseudoorbit}
\{x_0^p=0,x_1^p, \ldots, x_{R_p}^p= x_p,\ldots, x_{N_p-1}^p = 0\}.
\end{equation}
For the sequence $\{N_p=2R_p\}_{p\in\mathbb{N}}$, we associate a sequence of subsets $\{\Delta_p\}_{p\in\mathbb{N}}$ of $\mathbb{N}$ given by the Lemma \ref{freqhypercycliclemma}.

Bellow we describe an induction procedure to define a sequence of vectors $\{z_p\}_{p\in\mathbb{N}}$ of $X$ satisfying the following properties

\begin{itemize}
    \item[(a)] $\|z_p\|<\dfrac{1}{2^p}$ for every $p \in \mathbb{N}$;
    \item[(b)] if $m \in \Delta_p$ then $\displaystyle \left\|\sum_{0\leq q\leq p}T^{m+R_p}(z_q)-x_p\right\|<\frac{1}{2^p}$ for every $p \in \mathbb{N}$;
    \item[(c)] given $n,p \in \mathbb{N}$ with $n \notin \{ m, m+1, \ldots, m+ N_p-1: m \in \Delta_p\}$, then $\| T^{n} (z_p) \| < \dfrac{1}{2^p}$.
\end{itemize}

\noindent Since we assume $T^0 z_p=z_p$, the reader will notice that Property (a) is a consequence of Property (c), but we decided to make Property (a) explicit for the sake of clarity. Assume for now that such a  sequence $\{z_p\}_{p \in \mathbb{N}}$ exists. In this case the vector \[z=\sum_{p\in\mathbb{N}}z_p\] is well defined, since Property (a) implies $\|z\|=\|\sum z_p\|\leq\sum1/2^p=2$. We claim that $z$ is a frequently hypercyclic vector for $T$. Indeed, let $V$ be a non-empty open subset of $X$. Consider $w\in V$ and $\lambda>0$ such that $B(w,\lambda)\subset V$, and let $q_0\in\mathbb{N}$ be such that $1/2^{q_0}<\lambda/2.$ By the density of $\{x_p\}_{p\in\mathbb{N}}$, we can choose $x_p\in B(w,\lambda/2)$ with $p> q_0$. For each $m\in \Delta_p$, we have by Properties (b), (c) and item 2 of Lemma \ref{freqhypercycliclemma} that
$$
\|T^{m+R_p}(z)-x_p\| \leq \left\|\sum_{0\leq q\leq p}T^{m+R_p}(z_q)-x_p\right\|+\left\|\sum_{q>p}T^{m+R_p}(z_q)\right\| \leq
$$
$$
\frac{1}{2^p} + \sum_{q>p} \left\|T^{m+R_p}(z_q)\right\| \leq
\frac{1}{2^p}+\frac{1}{2^{p+1}}+\frac{1}{2^{p+2}}+\cdots = \frac{1}{2^{p-1}} < \dfrac{\lambda}{2}.
$$
Thus $T^{m+R_p}(z)\in B(x_p,\lambda/2)\subset V$ for all $m\in\Delta_p$, that is,
\[\underline{\emph{dens}}(\{n\in\mathbb{N}\;:\;T^n(z)\in V\})\geq \underline{\emph{dens}}(\Delta_p+R_p){=}\underline{dens}(\Delta_p)>0,\] where
$\Delta_p+R_p=\{m+R_p\;:\;m\in\Delta_p\}$ (see \cite{lowerdensity} for equality above). Proving, therefore, that $z$ is a frequently hypercyclic vector for $T$ and consequently $T$ is frequently hypercyclic.

We now obtain the vectors $z_p$ with Properties (a), (b) and (c). We use an induction procedure, as such the first step is to obtain $z_0$. For this end, consider the sequence $\{\gamma_n^0\}_{n\in\mathbb{N}}$ given by
\[\gamma_n^0=\left\{\begin{array}{ll}
x_k^{0} & \mbox{if } n=m+k \mbox{ with } m\in\Delta_0 \mbox{ and } k\in\{0,1,\ldots, N_0-1\}\\
0 &\mbox{if } n\in\mathbb{N}\setminus\{m,m+1,\ldots,m+N_0-1\;:\;m\in\Delta_0\}.
\end{array}\right.\] By the definition of \eqref{z0pseudoorbit} we have that $\{\gamma_n^0\}_{n\in\mathbb{N}}$ is a $\delta_0$-pseudo orbit. Thus the positive shadowing property guarantees the existence of  $z_0\in X$ such that
\[\|T^n(z_0)-\gamma_n^0\|<\eps_0, \mbox{ for every }n\in\mathbb{N}.\] This immediately implies that $\|z_0\|=\|T^0(z_0)-0\|=\|T^0(z_0)-\gamma_0^0\|<\eps_0$, and therefore $z_0$ satisfies Property (a). Notice that for each $m\in\Delta_0$, \[\|T^{m+R_0}(z_0)-\gamma_{m+R_0}^0\|=\|T^{m+R_0}(z_0)-x_{R_0}^0\|=\|T^{m+R_0}(z_0)-x_0\|<\eps_0\] which implies that $z_0$ satisfies Property (b). For every $n\in\mathbb{N}\setminus\{m,\ldots, m+N_0-1\;:\;m\in\Delta_0\}$, $\|T^n(z_0)-0\|=\|T^n(z_0)\|<\eps_0\leq 1$ and therefore $z_0$ satisfies Property (c).


As part of the induction procedure we will now assume that we have $z_1,z_2,\ldots,z_{p-1}$ and will now obtain $z_p$, $p\geq1$. For this end, define the sequence 
$\{\gamma_n^p\}_{n\in\mathbb{N}}$ as follows:
\[\gamma_n^p=\left\{\begin{array}{ll}
x_k^{p}-\dfrac{k}{R_p}T^{n}(z_0+z_1+\ldots+z_{p-1})& \mbox{if } n=m+k \mbox{ with } m\in\Delta_p\\
& \mbox{and } k\in\{0,1,\ldots,R_p\}\\
x_k^{p}- \dfrac{N_p-k}{R_p}T^n(z_0+z_1+\ldots+z_{p-1})& \mbox{if } n=m+k \mbox{ with } m\in\Delta_p\\
&\mbox{and } k\in\{R_p+1,\ldots, N_p-1\}\\
\\
0& \mbox{if }n\in\mathbb{N}\setminus\{m,\ldots, m+N_p-1\}\\
&\mbox{and } m\in\Delta_p.\\
\end{array}\right.\]
We shall see bellow that $\{\gamma_n^p\}_{n\in\mathbb{N}}$ is a $\delta_p$-pseudo orbit. This will guarantee the existence of a $z_p$ which $\eps_p$-shadows $\{\gamma_n^p\}$, that is,
\[\|T^n(z_p)-\gamma_n^p\|<\eps_p=\frac{1}{2^p}.\]
\noindent Since $\gamma_0^p=0$ follows that $\|z_p\|=\|T^0(z_p)-0\|<\eps_p=1/2^p$. Thus, $z_p$ satisfies the Property (a). 
\noindent Notice that with $z_p$ obtained this way we have, for each $m\in\Delta_p$, that \[\|T^{m+R_p}(z_0+z_1+\ldots+z_p)-x_p\|<\eps_p=\frac{1}{2^p}.\] Indeed, since $z_p$ shadows $\{\gamma_{n}^p\}$
\[\begin{array}{rcl}
\epsilon_p>\|T^{m+R_p}(z_p)-\gamma_{m+R_p}^p\|&=&\|T^{m+R_p}(z_p)-x_p+T^{m+R_p}(z_0+\ldots+z_{p-1})\|\\
&=&\|T^{m+R_p}(z_0+\ldots+z_p)-x_p\|.\\
\end{array}\]
Therefore, Property (b) is assured. Since $\gamma_{n}^p=0$ if $n \notin \{ m, m+1, \ldots, m+ N_p-1: m \in \Delta_p\}$ we have Property (c).


The last step remaining is to show that $\{\gamma_n^p\}_{n\in\mathbb{N}}$ is a $\delta_p$-pseudo orbit. First notice that from the definitions of $\Delta_p$ and $\Delta_q$ and item (2) from Lemma \ref{freqhypercycliclemma} follows that $\{k,\ldots, k+N_p\;:\;k\in\Delta_p\} \cap \{m,\ldots, m+N_q\;:\;m\in\Delta_q\} = \emptyset$ whenever $p\neq q$. This fact and Property (c) imply that if $n \in \{m,\ldots, m+N_p-1\;:\;m\in\Delta_p\}$, then
\begin{equation}\label{soma}
    \|T^{n}(z_0+z_1+\ldots+z_{p-1})\|=\left\|\sum_{0\leq q<p}T^n(z_q)\right\|\leq \sum_{0\leq q<p}\|T^n(z_q)\|\leq\sum_{0\leq q<p}\frac{1}{2^q}<2.
\end{equation} It is obvious that if $n,n+1\in\mathbb{N}\setminus\{m,\ldots,m+N_p-1\;:\;m\in\Delta_p\}$, then \[\|T(\gamma_n^p)-\gamma_{n+1}^p\|=0<\delta_p.\] If $n$ or $n+1$ belongs to $\{m,\ldots, m+N_p-1\;:\;m\in\Delta_p\}$, we have $5$ possibilities:

    \noindent\emph{Case 1:} $n\in\mathbb{N}\setminus\{m,\ldots, m+N_p-1\;:\;m\in\Delta_p\}$ and 
    $n+1\in\Delta_p$. 
    
    In this case, we have
    \[\gamma_n^p=0 \mbox{ and }\gamma_{n+1}^p=x_0^p-\frac{0}{R_p}T^{n+1}(z_0+\cdots+z_{p-1})=0.\]
    Hence,
    \[\|T(\gamma_n^p)-\gamma_{n+1}^p\|=0<\delta_p.\]
    %
    
    \noindent\emph{Case 2:} $n=m+k$ with $m\in\Delta_p$ and $k\in\{0,\ldots, R_p-1\}.$
    
    In this case, we have \[\gamma_n^p=x_k^{p}-\dfrac{k}{R_p}T^n(z_0+\cdots+z_{p-1}) \mbox{ and } \gamma_{n+1}^p=x_{k+1}^p-\frac{k+1}{R_p}T^{n+1}(z_0+\cdots+z_{p-1}).\] Then,
    \[\begin{array}{rcl}
    \|T(\gamma_n^p)-\gamma_{n+1}^p\| &=&\displaystyle \left\|T(x_k^{p}) - \dfrac{k}{R_p}T(T^n(z_0+\cdots+z_{p-1}))\right. -\\
    \\
    &&\displaystyle \;\;\;- \left.\left(x_{k+1}^p-\dfrac{k+1}{R_p}T^{n+1}(z_0+\cdots+z_{p-1})\right)\right\|\\
    \\
    &\leq &\displaystyle \|T(x_k^{p})-x_{k+1}^p\|+\frac{1}{R_p}\|T^{n+1}(z_0+\cdots+z_{p-1})\|\\ 
    \\
    &<&\dfrac{\delta_p}{2}+\dfrac{\delta_p}{4}\cdot 2\\
    \\
    & < &\delta_p.
    \end{array}\] The second inequality is assured by Equation (\ref{soma}).
    
    %
    \noindent\emph{Case 3:} $n=m+R_p$ with $m\in\Delta_p$.
    
   In this case, \[\gamma_n^p=x_{R_p}^p-\dfrac{R_p}{R_p}T^n(z_0+\cdots+z_{p-1})=x_p-T^n(z_0+\cdots+z_{p-1}) \mbox{ and }\]\[\gamma_{n+1}^p = x_{R_p+1}^p-\dfrac{N_p-(R_p+1)}{R_p}T^{n+1}(z_0+\cdots+z_{p-1}).\] Thus,
    \[\begin{array}{rcl}
    \|T(\gamma_n^p)-\gamma_{n+1}^p\| &=& \displaystyle\left\|T(x_p)-T^{n+1}(z_0+\cdots+z_{p-1})-x_{R_p+1}^p+\frac{R_p-1}{R_p}T^{n+1}(z_0+\cdots+z_{p-1})\right\|\\
    \\
    &\leq&\|T(x_p)-x_{R_p+1}^p\|+\dfrac{1}{R_p}\|T^{n+1}(z_0+\cdots+z_{p-1})\|\\
    \\
    &<&\dfrac{\delta_p}2+\dfrac{\delta_p}{4}\cdot 2\;\;\;\;\;\;\;\;\mbox{( by Eq. (\ref{soma}))}\\
    \\
    &=&\delta_p.\\
    \end{array}\] 
    
    \noindent\emph{Case 4:} $n=m+k$ with $m\in\Delta_p$ and $k\in\{R_p+1, R_p+2,\ldots, N_p-2\}$. It is analogous to case 2.
    
    
    \noindent\emph{Case 5:} $n=m+N_p-1$ with $m\in\Delta_p$. 
    
     In this case, we have
     \[\gamma_n^p=x_{N_p-1}^p-\frac{N_p-(N_p-1)}{R_p}T^n(z_0+\cdots+z_{p-1})=0-\dfrac{1}{R_p}T^n(z_0+\cdots+z_{p-1}) \mbox{ and }\gamma_{n+1}^p=0.\] By Equation (\ref{soma})
     \[\|T(\gamma_n^p)-\gamma_{n+1}^p\|=\dfrac{1}{R_p}\|T^{n+1}(z_0+\cdots+z_{p-1})\|<\frac{\delta_p}{4}\cdot 2<\delta_p.\]

\noindent This concludes that $\{\gamma_n^p\}_{n\in\mathbb{N}}$ is a $\delta_p$-pseudo orbit. 

\qed

\begin{corollary} \label{hyperimpliesfreq.hyper}
Let $X$ be a separable Banach space. If $T : X \to X$ is hypercyclic (or recurrent) and has the positive shadowing property then $T$ is frequently hypercyclic and topologically mixing.
\end{corollary}

\begin{corollary} \label{prodct of shadowing and CR}
If $T_i : X_i \to X_i$, for $1\leq i \leq n$, is a finite family of linear dynamical systems such that, for each $i$, $(T_i,X_i)$ satisfies the hypothesis of the Theorem \ref{gymstheorem}, then $T_1 \times \ldots \times T_n : X_1 \times \ldots X_n \to X_1 \times \ldots X_n$ is frequently hypercyclic and topologically mixing.
\end{corollary}

\noindent \textbf{Proof:} It follows from Theorem \ref{gymstheorem}, Proposition \ref{multiplechainrecurrence}, Theorem \ref{positiveshadowingmixing} and the fact that the shadowing property is preserved under cartesian product.

\qed

Consider a topological dynamical system $(X,T)$. By an \textbf{invariant measure with full support} for $(X,T)$ we mean a measure $\mu$ defined over the Borelian $\sigma$-algebra of $X$ such that $\mu(X)=1$, $\mu(U)>0$ for any open set $U \subset X$ and $\mu(T^{-1}A) = \mu(A)$ for any mensurable set $A$.

\begin{corollary}
Let $X$ be a separable and reflexive Banach space and $T : X \to X$ a chain transitive map that has the positive shadowing property then there is an invariant measure with full support for $T$.
\end{corollary}

\noindent \textbf{Proof:} It follows from Theorem \ref{gymstheorem} and the fact that frequently hypercyclic operators on reflexive Banach spaces admit an invariant measure with full support \cite{invariant}.

\qed

\textbf{Devaney chaotic} systems are those that are transitive and have a dense set of periodic points \cite{devaney2}. The next corollary gives a suficient condition for a system to have this property. 

\begin{corollary}
Let $X$ be a separable Banach space. If $T : X \to X$ has a dense set of periodic points and has the positive shadowing property then $T$ is Devaney chaotic.
\end{corollary}

\noindent \textbf{Proof:} Since $CR(T)$ is closed, and the set of periodic points is dense, then $X=CR(T)$. Therefore, Theorem A implies that $T$ is frequently hypercyclic and topologically mixing.

\qed




\begin{corollary}
If $X$ is a separable Hilbert space and $T$ is an unitary operator, then $T$ does not have the positive shadowing property.
\end{corollary} 

\noindent \textbf{Proof:} By proposition \ref{unitaryimplieschaintransitivity} unitary operators are chain transitive, but since $\|T\| = 1$ they may never be hypercyclic. Therefore, by Theorem A these operators cannot have positive shadowing.

\qed





Given a topological dynamical system $(Y,f)$, a point in $x \in Y$ is said to be a \textbf{non-wandering point} if, for every open set $B$ that contains $x$, and for every $N \in \mathbb{N}$, there is $n>N$, such that $f^{n}(B) \cap B \neq \emptyset$. If $(X,T)$ is a linear dynamical system, linearity implies that the origin is always a non wandering point of $X$. We call the set of non wandering points of \textbf{non wandering set}. The complement of the non wandering set is the \textbf{wandering set}.

It is not difficult to see that if $T$ is recurrent, then every $x \in X$ is non wandering, and if $T$ has a dense set of non wandering points, then $T$ is recurrent.

The next result can be easily proven (it is similar to Proposition 6 of \cite{Messaoudi}), but we state it since it is needed in the proof of Theorem B.

\begin{proposition} \label{shadowingforsubspaces}
Let $T$ be an operator on a Banach space $X$. Suppose that $X = M \oplus N$, where $M$ and $N$ are closed $T$-invariant subspaces of $X$. Then $T$ has positive shadowing property if, and only if, $T|_M$ and $T|_N$ both have positive shadowing property.
\end{proposition}


\noindent \textbf{Proof of Theorem \ref{theo:nonwandering}:} It is easy to see that the non-wandering set is contained in the set of chain recurrent elements. For the contrary inclusion if $x$ is chain recurrent, and $U$ is an open set that contains $x$, then there is a pseudo-orbit that goes from $x$ to $x$ and then to $x$ again and so on, and the shadowing property tells us that such pseudo orbit can be shadowed by a real orbit. This proves that $x$ is a non-wandering point. Since $\Omega$ is equal to the chain recurrent set, then Corollary \ref{chainspace} tells us that  $\Omega$ is a closed and invariant subspace of $X$. This allows us to obtain the first two conclusions.

By item 1, we have $\Omega=CR(T)$.  By Corollary \ref{chainspace}, $\Omega$ is $T$-invariant. This implies that $\Omega^\perp$ is $T^*$-invariant, consequently, it is $T$-invariant, because $T=T^*$. Since $X$ is a Hilbert space and $\Omega$ is a closed subspace of $X$ then $X=\Omega\oplus\Omega^{\perp}$. Hence, by Proposition \ref{shadowingforsubspaces}, $T|_{\Omega}$ has positive shadowing property. Corollary \ref{chainrecurrentselfadjoint} guarantees that $T|_{\Omega}$ is chain transitive. Since $T|_{\Omega} : \Omega \to \Omega$ has positive shadowing property and  is chain transitive, Theorem \ref{gymstheorem} gives that $T|_{\Omega}$ is topologically mixing and frequently hypercyclic.

\qed

\section{Open Questions}

In this section we leave some open questions for the reader. Question \ref{refereequestion} was kindly offered to us by the anonymous referee. Originally, this question only addressed the shadowing property, but we decided to expand it for chain recurrence as well.

\begin{question}
Is there any simple criterion to decide if a system is chain recurrent?
\end{question}

Since every hypercyclic operator is chain recurrent, then Kitai's Hypercyclic Criterion \cite{dynamicsoflinearoperators} is a simple criterion for chain recurrence. But since there are operators that are chain recurrent but not hypercyclic, e.g. the identity operator, it would be nice to find a tighter criterion.

\begin{question}
Is there any simple criterion to decide if a system has positive shadowing?
\end{question}

This question is not new and was addressed by other authors for the shadowing property. It seems that the notion of generalized hyperbolicity \cite{generalized} captures much of the essence of the shadowing property and could be equivalent to it. In view of the results presented in this text positive shadowing might be a more relevant property than shadowing itself, and therefore worthy of a similar search for an equivalent notion. Example \ref{eigen1andshadowing} shows that when $T$ has a right inverse that is a proper contraction, then $T$ will have positive shadowing.

\begin{question} \label{refereequestion}
Let $X$ be a normed vector space and $T:X \to X$ a linear operator with the positive shadowing property (and or chain recurrence), and let $Y \subset X$ be a closed invariant subspace. Under which hypotheses one may guarantee that the operator $T|_Y$ has the shadowing property (and or chain recurrence)?
\end{question}

Proposition \ref{shadowingforsubspaces} and Corollary \ref{CRforsubspaces} provide partial answers to this question. New results in this direction may provide a proof of the last item of Theorem B under more general assumptions than $X$ being Hilbert and $T$ being self-adjoint.

\begin{question}
Are the conclusions of Theorem A strong enough to imply the hypothesis?
\end{question}

More precisely does every frequently hypercyclic and topologically mixing linear dynamical system have positive shadowing? At the moment, the authors have no plausible argument that favors such conclusion, but we are unable to provide a counter-example to this claim. Based on Corollary \ref{prodct of shadowing and CR}, one may also ask a weaker version of this question, which is: if $(X,T)$ is a linear dynamical system such that any finite product of $T \times \ldots \times T$ is topologically mixing and frequently hypercyclic, then would this imply that $T$ has positive shadowing?


\vspace{0.5cm}

\noindent \textit{Acknowledgements: We would like to thank Prof. Udayan B. Darji for helpful comments on an earlier version of this text. We would also like to show our deep appreciation for the careful review and interesting suggestions given by the anonymous referee, which include among others, lemma \ref{refereelemma} and its proof, Corollary \ref{properdilation}, example \ref{refereeexample} and question \ref{refereequestion}. M.B.A. was supported by CAPES. R.V. was partially supported by CNPq.}


\end{document}